\documentclass [11pt]{article}
\usepackage[]{amsmath,amssymb,amsfonts,latexsym,amsthm,enumerate}
\usepackage{mathrsfs}
\usepackage[numeric,initials,nobysame]{amsrefs}

\long\def\symbolfootnote[#1]#2{\begingroup%
\def\thefootnote{\fnsymbol{footnote}}\footnote[#1]{#2}\endgroup}

\def\R{\mathbb{R}}

\def\prob#1{\textrm{Pr}(#1)}

\date{}
\title{A concentration result with application to subgraph count}
\author{{Guy Wolfovitz\thanks{Department of Computer Science, 
Haifa University, Haifa, Israel. Email address:
{\tt gwolfovi@cs.haifa.ac.il}.}}}
\newtheorem{theorem}{Theorem}[section]
\newtheorem{lemma}[theorem]{Lemma}

\newtheorem{definition}{Definition}

\newtheorem{proposition}[theorem]{Proposition}

\newtheorem{problem}[theorem]{Problem}

\renewcommand{\epsilon}{\varepsilon}

\DeclareMathOperator{\var}{Var}
\DeclareMathOperator\expec{{\mathbb E}}

\DeclareMathOperator\codeg{{codeg}}

\newtheoremstyle{upright}%
        {8pt plus2pt minus4pt}%
        {8pt plus2pt minus4pt}%
        {\upshape}%
        {}%
        {\bfseries}%
        {:}%
        {1em}%
        {}%
\theoremstyle{upright}

\newcommand{\ignore}[1]{}
\begin{document}

\maketitle

\begin{abstract}
Let $H = (V,E)$ be a $k$-uniform hypergraph with a vertex set $V$ and an edge
set $E$.  Let $V_p$ be constructed by taking every vertex in $V$ independently
with probability~$p$.  Let $X$ be the number of edges in $E$ that are contained
in $V_p$.  We give a condition that guarantees the concentration of $X$ within
a small interval around its mean.  The applicability of this result is
demonstrated by deriving new sub-Gaussian tails for the number of copies of
small complete and complete bipartite graphs in the binomial random graph,
extending results of Ruci\'nski and Vu.
\end{abstract}

%%%%%%%%%%%%%%%%%%%%%%%%%%%%%%%%%%%%
\section{Introduction} \label{sec:1}
%%%%%%%%%%%%%%%%%%%%%%%%%%%%%%%%%%%%
%
Let $H = (V, E)$ be a hypergraph, where $V$ is a set of $n$ vertices and $E
\subseteq 2^V$ is a set of $m$ edges.  Assume that $H$ is $k$-uniform, that is
every edge in $E$ contains exactly $k$ vertices.  Let $0 < p <  1$ and let
$V_p$ be a random set of vertices constructed by taking every vertex in $V$
independently with probability $p$.  Let $H_p = (V_p, E_p)$ be the hypergraph
with vertex set $V_p$ and edge set $E_p$, where $e \in E_p$ if and only if $e
\in E$ and $e \subseteq V_p$.  Let $X := |E_p|$ count the number of edges of
$H_p$.  The main aim of this paper is to provide a condition which guarantees
the concentration of $X$ within a small interval around its mean. 

Before presenting the main result, let us give our motivation for studying the
random variable $X$.  Let $G$ be a fixed graph with $v_G$ vertices and $e_G$
edges.  Let $K_N$ denote the complete $N$-vertex graph.  Let $H_G$ be the
$e_G$-uniform hypergraph with a vertex set consisting of all edges in $K_N$ and
with an edge set consisting of all copies of $G$ in $K_N$. If we take $H = H_G$
and then let $X_G = X$, then $X_G$ counts the number of copies of $G$ in the
binomial random graph $G(N,p)$ (that is, the graph that is constructed by
taking every edge in $K_N$ independently with probability $p$).  The study of
$X_G$ is a classical topic in the theory of random graphs~(see
e.g.,~\cites{B85, JLR}). Here we are interested in the following problem which
was studied by Vu~\cites{Vu01,Vu02} and later also mentioned by
Kannan~\cite{K90}.
\begin{problem} \label{problem}
 Determine for which $p$ and $\lambda$ does $X_{G}$ have the sub-Gaussian tails
\begin{eqnarray} 
\label{eq:q} \prob{|X_G - \expec(X_G)| \ge \lambda \var(X_G)^{1/2}} \le e^{-c_G
\lambda^2}, 
\end{eqnarray}
where $c_G$ is a positive constant that depends only on $G$.
\end{problem}
Define $\rho_1 = \rho_1(G) := v_G/e_G$ and $\rho_2= \rho_2(G)  :=
(v_G-2)/(e_G-1)$.  Using our main result we prove the following.
\begin{theorem}  \label{thm:sub}
If $G$ is a complete  or a complete bipartite graph with $e_G \ge 3$, then for
every positive constant $c_1$ there is a positive constant $c_2 = c_2(G, c_1)$
such that~(\ref{eq:q}) holds provided $N^{-\rho_1 + c_1} \le p \le N^{-\rho_2 -
c_1}$ and $8 \ln N \le \lambda \le N^{c_2}$.
\end{theorem}
Suppose that $G$ is a complete  or a complete bipartite graph.  A result of
Vu~\cite{Vu02} implies that for every positive constant $c_1$ there are
positive constants $c_3 = c_3(G)$ and $c_4 = c_4(G, c_1)$ such
that~(\ref{eq:q}) holds provided $p \ge N^{-\rho_2 + c_1 + c_3}$ and $c_4^{-1}
\le \lambda \le N^{c_4}$.  Furthermore, when $G$ is a complete graph, one can
take $c_3 = 0$.  In addition, a result of Ruci\'nski~\cite{Ru88} implies
that~(\ref{eq:q}) holds provided $1/2 \ge p = \omega(N^{-\rho_1})$ and
$\lambda$ is constant.  Observe that Theorem~\ref{thm:sub} extends both of
these results. In particular, when $G$ is a complete graph
Theorem~\ref{thm:sub} in a sense complements Vu's result, as for every positive
constant $c_1$ the former handles the case $N^{-\rho_1 + c_1} \le p \le
N^{-\rho_2 - c_1}$ while the latter handles the case $p \ge N^{-\rho_2 + c_1}$.

\subsection{Main result}
In order to state our main result, we need some definitions and notation.  The
degree of a vertex $v$ of a given hypergraph is the number of edges of the
hypergraph that contain $v$.  The co-degree of two distinct vertices $u, v$ of
a given hypergraph is the number of edges of the hypergraph that contain both
$u$ and $v$.  Denote by $\deg_p(v)$ the degree of a vertex $v$ of $H_p$.
Denote by $\codeg_p(u,v)$ the co-degree of two distinct vertices $u, v$ of
$H_p$.  Denote by $\Delta$ (resp. $\delta$) the maximum (resp. minimum) degree
of a vertex of $H$.  Denote by $\Delta_2$ the maximum co-degree of two distinct
vertices of $H$.
The following definition provides the condition which will be shown to imply
the concentration of $X$.
\begin{definition} \label{def:nice} Say that $(H,p,\lambda,\Gamma, b)$ is
\emph{nice} if the following properties hold.
\begin{itemize}
\item[(P1)] $p \le 10^{-3}$, $k \ge 3$ is constant and $n \ge n_0$ for a
sufficiently large constant $n_0 = n_0(k)$; 
\item[(P2)] $(p^k m)^{1/2} \ge \max\{ \ln n, \lambda\}$;
\item[(P3)] $\Delta_2 \le \delta \ln^{-3} n$;
\item[(P4)] Let $p \le q < 1$. With probability at least $1 - e^{-b \lambda^2}$
we have:
\begin{itemize} 
\item $\forall v \in V_q: \deg_q(v) \le \max  \{2 q^{k-1} \Delta, \Gamma\}$; 
\item $p^{1/2} q^{k-3/2} \Delta^2 n \ln n \ge m \implies \forall u, v, w \in
V_q: \codeg_q(u,v) \le \deg_q(w) \ln^{-3}n$. 
\end{itemize}
\end{itemize}
\end{definition}
Let us briefly discuss the condition that $(H, p, \lambda, \Gamma, b)$ is nice.
Property~(P1) is clear. Property~(P2) is equivalent to saying that the
expectation of $X$ is sufficiently large -- that it is lower bounded by
$\ln^2n$ and $\lambda^2$. Property~(P3) says that the maximum co-degree of $H$
is sufficiently small with respect to the minimum degree of $H$, and
property~(P4) says that with a sufficiently large probability this also holds
for $H_q$, provided $q$ is sufficiently large. Lastly, property~(P4) also says
that with a sufficiently large probability the maximum degree of $H_q$ behaves
roughly as we expect it to.

%Our main result follows.
%
\begin{theorem}[Main result]  \label{thm:main}
If $(H, p, \lambda, \Gamma, b)$ is nice then $$\prob{|X - \expec(X)| > (\ln n +
\lambda) \expec(X)^{1/2}} \le 2 (\gamma_1 + \gamma_2) \ln n,$$ where, for some
positive constant $b_k$ that depends only on $k$,
\begin{eqnarray*}
\gamma_1 &=& e^{-b \lambda^2} + 2 e^{-b_k \lambda^2} + 2 \exp \bigg(- \frac{b_k
m} {p^{k-1} \Delta^2 n \ln n}\bigg); \\ 
\gamma_2 &=& 2\exp\bigg(-\frac{ b_k pn } { \ln^5 n } \bigg) +
2\exp\bigg(-\frac{ b_k p^k m } { \Gamma^2 \ln^6 n } \bigg).
\end{eqnarray*}
\end{theorem}
Here, a rather simple application of our main result is sketched. Suppose that
$H$ is the hypergraph $H_{G}$ that was defined above, with $G$ being a
triangle.  In that case, $X_G = X$ counts the number of triangles in $G(N, p)$.
We have $n = \binom{N}{2}$, $m = \binom{N}{3}$, $\Delta = \delta = N-2$,
$\Delta_2 = 1$ and $k=3$. Assume that $n \ge n_0$ for a sufficiently large
constant $n_0$. Let $p = N^{-1} \ln^{50}N$ and $\lambda = \ln^{10} N$.  It is
easy to show using Chernoff's bound~(see e.g.,~\cite{JLR}) that $(H, p, 0.25
\lambda, \lambda^2, b)$ is nice for some positive constant $b$.  Using this,
one can easily see that Theorem~\ref{thm:main} implies that the probability
that $X_G$ deviates from its expectation by more than $(\ln n + 0.25 \lambda)
\expec(X_G)^{1/2}$ is at most $e^{- c_G \lambda^2}$ for some positive constant
$c_G$.  Now, clearly $\ln n \le 0.25 \lambda$ and in addition, for our choice
of $p$ we have $0.5 \expec(X_G)^{1/2} < \var(X_G)^{1/2}$.  Thus we infer the
following sub-Gaussian behavior: the probability that $X_G$ deviates from its
expectation by at least $\lambda \var(X_G)^{1/2}$ is at most $e^{- c_G
\lambda^2}$.

We note that for some range of the parameters (e.g., in some cases where
$\Delta$ is not bounded and $\var(X)$ equals up to a constant to $\expec(X)$ --
as is implicitly the case in the example above and in the proof of
Theorem~\ref{thm:sub}), Theorem~\ref{thm:main} does not follow directly from
other known concentration results such as Azuma's inequality or Talagrand's
inequality~(see e.g.,~\cite{JR02}), Kim and Vu's inequalities~(see
e.g.,~\cites{Vu02}) or the more recent result of Kannan~\cite{K90}. 
In addition, we should note that a weaker version of Theorem~\ref{thm:main} has
been used implicitly by the author in~\cite{Guy}, in order to prove
Theorem~\ref{thm:sub} for the special case where $G$ is a triangle.
In fact, in that special case it turns out that better bounds for $p$ and
$\lambda$ can be given.

\subsection{A probabilistic tool}
The proof of Theorem~\ref{thm:main} is based on an iterative application of
McDiarmid's inequality~\cite{McD}, which we state now.  Let $\alpha_1, \alpha_2,
\ldots, \alpha_l$ be independent random variables with $\alpha_i$ taking values
in a set $A_i$. Let $f : \prod_{i=1}^l A_i \to \R$ satisfy the following
Lipschitz condition: if two vectors $\alpha, \alpha' \in \prod_{i=1}^l A_i$
differ only in the $i$th coordinate, then $|f(\alpha) - f(\alpha')| \le a_i$.
McDiarmid's inequality states that the random variable $W = f(\alpha_1,
\alpha_2,\ldots, \alpha_l)$ satisfies for any $t \ge 0$,
\begin{eqnarray*} 
\prob{|W - \expec(W)| \ge t} \le 2\exp\bigg(-\frac{2t^2}{\sum_{i=1}^l
a_i^2}\bigg).  
\end{eqnarray*}

%%%%%%%%%%%%%%%%%%%%%%%%%%%%%%%%%%%
\subsection{Structure of the paper}
%%%%%%%%%%%%%%%%%%%%%%%%%%%%%%%%%%%
%
In Section~\ref{sec:2} we state a technical lemma (Lemma~\ref{lemma}) and use
it to prove Theorem~\ref{thm:main}. That technical lemma is proved in
Section~\ref{sec:3}. Finally, in Section~\ref{sec:4} we use
Theorem~\ref{thm:main} to derive Theorem~\ref{thm:sub}.

%%%%%%%%%%%%%%%%%%%%%%%%%%%%%%%%%%%%%%%%%
\section{Proof of Theorem~\ref{thm:main}} \label{sec:2}
%%%%%%%%%%%%%%%%%%%%%%%%%%%%%%%%%%%%%%%%%
%
Assume that $(H, p, \lambda, \Gamma, b)$ is nice and note that if follows
from~(P1)~and~(P2) that $1/n \le p \le 10^{-3}$. 

Consider the following alternative, iterative definition of the random set
$V_p$.  Let $\epsilon \in [10^{-6}, 10^{-3}]$ and let $I \le \ln n$ be an
integer such that $\epsilon^I = p$.  Define a sequence of sets $(V_i)_{i=0}^I$
as follows.  Let $V_0 := V$. Given $V_i$, construct $V_{i+1}$ by taking every
vertex $v \in V_i$ independently with probability $\epsilon$.  End upon
obtaining $V_I$.  (Note that this definition does not introduce any ambiguity,
as we've defined $V_p$ in the introduction only for $0 < p < 1$.) Observe that
for every integer $0  <  i \le I$, $V_i$ has the same distribution as
$V_{\epsilon^i}$. In particular, since $\epsilon^I = p$, we have that $V_I$ has
the same distribution as $V_p$.

We need the following definitions, notation and lemma. Let $0 \le i \le I$ be
an integer.  Let $H_i = (V_i, E_i)$ be the hypergraph with vertex set $V_i$
and edge set $E_i$, where $e \in E_i$ if and only if $e \in E$ and $e \subseteq
V_i$.  Let $X_i := |E_i|$ be the number of edges of $H_i$ and note that $X_I =
X$.  For a vertex $v \in V_i$, let $\deg_i(v)$ be the degree of $v$ in $H_i$.
For a vertex $v \notin V_i$, let $\deg_i(v) := 0$.  For two distinct vertices
$u, v \in V_i$, let $\codeg_i(u,v)$ be the co-degree of $u$ and $v$ in $H_i$.
Lastly, let $x \pm y$ denote the interval $[x-y, x+y]$.
\begin{lemma} \label{lemma}
For every integer $0 \le i < I$ the following holds.  Assume that $(H, p,
\lambda, \Gamma, b)$ is nice, and in addition,
\begin{enumerate}
\item[(i)] $X_i \in \epsilon^{ki}m \pm \big( i(p^km)^{-1/2}\epsilon^{ki}m +
\lambda (\epsilon^{(k+1)i} m p^{-1})^{1/2}\big)$;
\item[(ii)] $\sum_{v \in V} \deg_i(v)^2 \le \epsilon^{(2k-1)i} \Delta^2 n ( 1 +
3i \ln^{-2} n ) + 6 k \epsilon^{(k+1/2)i} m p^{-1/2}$; 
\item[(iii)] $\forall v \in V_i: \deg_i(v) \le \max\{2 \epsilon^{(k-1)i}\Delta,
\Gamma\}$;
\item[(iv)] $p^{1/2} \epsilon^{(k-3/2)i} \Delta^2 n \ln n \ge m \implies
\forall u, v, w \in V_i: \codeg_i(u,v) \le \deg_i(w) \ln^{-3} n$.
\end{enumerate}
Then the following two items hold respectively with probabilities at least
$1-\gamma_1$ and $1-\gamma_2$, where $\gamma_1$ and $\gamma_2$ are as given in
the statement of Theorem~\ref{thm:main}:
\begin{itemize} 
\item $X_{i+1} \in \epsilon^{k(i+1)}m \pm \big( (i+1)(p^km)^{-1/2}
\epsilon^{k(i+1)}m + \lambda (\epsilon^{(k+1)(i+1)} m p^{-1})^{1/2}\big)$; 
\item $\sum_{v \in V} \deg_{i+1}(v)^2 \le \epsilon^{(2k-1)(i+1)} \Delta^2 n ( 1
+ 3(i+1) \ln^{-2} n ) + 6 k \epsilon^{(k+1/2)(i+1)} m p^{-1/2}$.  
\end{itemize} \end{lemma}

We prove Theorem~\ref{thm:main}.  We claim that for every integer $0 \le j \le
I$, the following holds: with probability at least $1 - 2 j (\gamma_1 +
\gamma_2)$, the four preconditions~(i)~through~(iv) in Lemma~\ref{lemma} hold
for $i=j$.  The proof of this claim is by induction. It is easy to verify
that~(i)~through~(iv) hold for $i=0$ with probability $1$ (here we use
property~(P3)) and so the claim holds for $j=0$.  Let $0 \le j < I$ be an
integer and assume that the claim holds for~$j$. By the induction hypothesis
and Lemma~\ref{lemma} we have that~(i)~and~(ii) hold for $i = j + 1$ with
probability at least $1 - 2 j (\gamma_1 + \gamma_2) - \gamma_1 - \gamma_2$.
From~(P4) we have that~(iii)~and~(iv) hold for $i = j + 1$ with probability at
least $1 - \gamma_1$.  Therefore, as needed, we can conclude
that~(i)~through~(iv) hold for $i = j + 1$ with probability at least $1 - 2 j
(\gamma_1 + \gamma_2) - 2\gamma_1 - \gamma_2 \ge 1 - 2 (j+1) (\gamma_1 +
\gamma_2)$.

By the above claim and since $I \le \ln n$, we have that with probability at
least $1 - 2 ( \gamma_1 + \gamma_2) \ln n$,
$$X_I \in \epsilon^{kI}m \pm \Big(I(p^km)^{-1/2} \epsilon^{kI}m + \lambda
(\epsilon^{(k+1)I} m p^{-1})^{1/2}\Big) \subseteq \expec(X) \pm (\ln n +
\lambda)  \expec(X)^{1/2},$$
where the last containment follows since $\epsilon^I = p$, $\expec(X) = p^k m$
and $I \le \ln n$. This gives the theorem.

%%%%%%%%%%%%%%%%%%%%%%%%%%%%%%%%%%%%%%%%
\section{Proof of Lemma~\ref{lemma}} \label{sec:3}
Let $0 \le i < I$ be an integer. Assume that $(H, p, \lambda, \Gamma, b)$ is
nice and that we are given $H_i$ so that the preconditions~(i) through~(iv) in
Lemma~\ref{lemma} hold.  We prove below that the first consequence in
Lemma~\ref{lemma} holds with probability at least $1 - \gamma_1$ and that the
second consequence holds with probability at least $1 - \gamma_2$.  
Let $b_k$ be a sufficiently small constant that depends only on $k$, chosen so
as to satisfy our inequalities below.
For future reference we record the following useful inequality, which may or
may not be valid (depending on~$i$):
\begin{eqnarray} \label{eq:iq} 
p^{1/2} \epsilon^{(k-3/2)i} \Delta^2 n \ln n \ge m.
\end{eqnarray}

%%%%%%%%%%%%%%%%%%%%%%%%%%%%%%
\subsection{First consequence}
We have $\expec(X_{i+1}) = \epsilon^k X_i$. Thus, using precondition~(i)
and~(P1) (specifically the fact that $k \ge 3$),
\begin{eqnarray*}
\expec(X_{i+1}) &\in& \epsilon^{k(i+1)} m \pm \Big(i (p^km)^{-1/2}
\epsilon^{k(i+1)} m + \epsilon^k \lambda (\epsilon^{(k+1)i} m p^{-1})^{1/2}
\Big) \\
&\subseteq& \epsilon^{k(i+1)} m \pm \Big(i (p^km)^{-1/2} \epsilon^{k(i+1)} m +
\epsilon \lambda (\epsilon^{(k+1)(i+1)} m p^{-1})^{1/2} \Big).
\end{eqnarray*}
It remains to upper bound the probability that $X_{i+1}$ deviates from its
expectation by more than 
$$t_1 := (p^km)^{-1/2}\epsilon^{k(i+1)}m + (1 - \epsilon) \lambda
(\epsilon^{(k+1)(i+1)}m p^{-1})^{1/2}.$$
Every vertex $v \in V_i$ has an outcome which is either the event that $v \in
V_{i+1}$ or not.  Clearly $X_{i+1}$ depends on the outcomes of the vertices in
$V_i$ and changing the outcome of a single vertex $v \in V_i$ can change
$X_{i+1}$ by at most an additive factor of $\deg_i(v)$. Using McDiarmid's
inequality, the fact that $\sum_{v \in V_i} \deg_i(v)^2 = \sum_{v \in V}
\deg_i(v)^2$, precondition~(ii) and the fact that $i < I \le \ln n$, we get
\begin{eqnarray} \label{eq:q1}
\prob{|X_{i+1} - \expec(X_{i+1})| > t_1} &\le& 2 \exp\bigg(-\frac{t_1^2} { 6 k
\epsilon^{(2k-1)i}\Delta^2n + 6 k \epsilon^{(k+1/2)i} m p^{-1/2} }\bigg).
\end{eqnarray}
Suppose that~(\ref{eq:iq}) holds.  Then the denominator of the exponent
in~(\ref{eq:q1}) is at most $12 k \epsilon^{(2k-1)i} \Delta^2 n \ln n$.  In
addition we have $t_1 \ge (p^km)^{-1/2} \epsilon^{k(i+1)}m$. Thus,
from~(\ref{eq:q1}) we get
\begin{eqnarray} \nonumber \prob{|X_{i+1} - \expec(X_{i+1})| > t_1} & \le & 2
\exp\bigg(-\frac{ (p^km)^{-1} \epsilon^{2k(i+1)}m^2 }{ 12 k \epsilon^{(2k-1)i}
\Delta^2 n \ln n}\bigg) \\ 
\nonumber &=& 2 \exp\bigg(-\frac{ \epsilon^{i + 2k} m }{ 12 k p^k \Delta^2 n
\ln n}\bigg) \\ 
\label{eq:g1}  &\le& 2 \exp\bigg(- \frac{b_k m}{ p^{k-1}
\Delta^2 n \ln n}\bigg), \end{eqnarray}
where the last inequality follows since $p < \epsilon^i$.

Now suppose that~(\ref{eq:iq}) does not hold.  Then the denominator of the
exponent in~(\ref{eq:q1}) is at most $12 k \epsilon^{(k+1/2)i} m p^{-1/2}$.  We
also have that $t_1 \ge 0.5 \lambda (\epsilon^{(k+1)(i+1)} m p^{-1})^{1/2}$ and
$p < \epsilon^i$.  Therefore, using~(\ref{eq:q1}) we get
\begin{eqnarray} \label{eq:g2} 
\prob{|X_{i+1} - \expec(X_{i+1})| > t_1}  \le 2 \exp\bigg(-\frac{ 0.25
\lambda^2 \epsilon^{(k+1)(i+1)} m p^{-1}} {12 k \epsilon^{(k+1/2)i} m
p^{-1/2}}\bigg) \le 2 e^{-b_k \lambda^2}.
\end{eqnarray}
We conclude from~(\ref{eq:g1})~and~(\ref{eq:g2}) that the first consequence
holds with probability at least $1 - \gamma_1$.

%%%%%%%%%%%%%%%%%%%%%%%%%%%%%%%
\subsection{Second consequence}
Define
\begin{eqnarray*} 
\eta := \left\{ \begin{array}{ll} 
k \ln^{-3} n & \text{if~(\ref{eq:iq}) holds}, \\ 
1 & \text{otherwise.}  \end{array} \right.
\end{eqnarray*}
Let $Y := \sum_{v \in V}  \deg_{i+1}(v)^2$.  We start by upper bounding
$\expec(Y)$.  For that we need the next fact.
\begin{proposition} \label{claim:1} For all $v \in V$, $\expec(\deg_{i+1}(v)^2)
\le (\epsilon^{2k-1} +  \epsilon^{k+1} \eta) \deg_i(v)^2 + \epsilon^k
\deg_i(v)$.  \end{proposition}
\begin{proof}
If $v \notin V_i$ then the proposition holds since in that case we trivially
have $\deg_i(v) = 0$ and $\deg_{i+1}(v) = 0$.
Assume that $v \in V_i$.  Let $\deg'_{i+1}(v)$ be the number of edges $e \in
E_i$ with $v \in e$, such that $e - \{v\} \subseteq V_{i+1}$. Note that
$\deg_{i+1}(v) = \deg'_{i+1}(v) \cdot {\bf 1}[v \in V_{i+1}]$, where ${\bf 1}[v
\in V_{i+1}]$ is the indicator function for the event that $v \in V_{i+1}$.
Since $\deg'_{i+1}(v)$ is independent of ${\bf 1}[v \in V_{i+1}]$ and $\prob{v
\in V_{i+1}} = \epsilon$, we have $\expec(\deg_{i+1}(v)^2) = \epsilon
\expec(\deg'_{i+1}(v)^2)$.

Clearly $\expec(\deg'_{i+1}(v)) = \epsilon^{k-1} \deg_i(v)$. It is also easy to
see that $\var(\deg'_{i+1}(v)) \le \epsilon^{k-1} \deg_i(v) + \epsilon^{k}
\sum_{e, e'} 1$, where the sum $\sum_{e, e'}$ ranges over all pairs of distinct
edges $e, e' \in E_i$ that contain $v$ and share at least $2$ vertices. By
precondition~(iv) we can bound $\sum_{e, e'} 1$ from above by $\eta
\deg_i(v)^2$.  Therefore, 
\begin{eqnarray*} \expec(\deg_{i+1}(v)^2) &=& \epsilon \expec(\deg'_{i+1}(v)^2)
\\ &=& \epsilon \expec(\deg'_{i+1}(v))^2 + \epsilon \var(\deg'_{i+1}(v)) \\ &
\le& \epsilon^{2k-1} \deg_i(v)^2 + \epsilon^{k+1} \eta \deg_i(v)^2 + \epsilon^k
\deg_i(v).  \end{eqnarray*}
\end{proof}

\begin{proposition} \label{claim:2} 
$\expec(Y) \le \epsilon^{(2k-1)(i+1)} \Delta^2 n ( 1 + (3i+2) \ln^{-2} n ) + 5
k \epsilon^{(k+1/2)(i+1)} m p^{-1/2}$.
\end{proposition}
\begin{proof}
By Proposition~\ref{claim:1},
\begin{eqnarray*}
\expec(Y) = \sum_{v \in V} \expec(\deg_{i+1}(v)^2) \le \sum_{v \in V}
(\epsilon^{2k-1} + \epsilon^{k+1} \eta) \deg_i(v)^2 + \epsilon^{k} \deg_i(v).
\end{eqnarray*}
By precondition~(ii) and~(P1) (specifically the fact that $k \ge 3$) and since
$6 \epsilon^{3/2} \le 1$,
\begin{eqnarray*} 
\sum_{v \in V} \epsilon^{2k-1} \deg_i(v)^2 &\le& \epsilon^{(2k-1)(i+1)}
\Delta^2 n (1 + 3i \ln^{-2} n ) + 6 k \epsilon^{(k+1/2)i + 2k - 1} m p^{-1/2}
\\ &\le& \epsilon^{(2k-1)(i+1)} \Delta^2 n (1 + 3i \ln^{-2} n ) + k
\epsilon^{(k+1/2)(i+1)} m p^{-1/2}.
\end{eqnarray*}
Note that every edge in $E_i$ is counted exactly $k$ times in the sum $\sum_{v
\in V} \deg_i(v)$ and so $\sum_{v \in V} \deg_i(v) = k X_i$. Moreover,
precondition~(i),~(P2) and the facts that $p  <  \epsilon^i$ and $i \le \ln n$
give us that $X_i \le 3 \epsilon^{ki} m$.  Hence $\sum_{v \in V} \deg_i(v) \le
3 k \epsilon^{ki}m$.  Thus, since $1 \le \epsilon^{i+1} p^{-1}$,
\begin{eqnarray*} 
\sum_{v \in V} \epsilon^k \deg_i(v) \le 3 k \epsilon^{k(i+1)}m \le 3 k
\epsilon^{(k+1/2)(i+1)} m p^{-1/2}. 
\end{eqnarray*}
Given the above, in order to complete the proof it is enough to show that
\begin{eqnarray} \label{eq:todo} 
\sum_{v \in V} \epsilon^{k+1} \eta \deg_i(v)^2 &\le& 2\epsilon^{(2k-1)(i+1)}
\Delta^2n \ln^{-2} n + k \epsilon^{(k+1/2)(i+1)} m p^{-1/2}.
\end{eqnarray}

Suppose that~(\ref{eq:iq}) holds.  In that case $\eta = k \ln^{-3} n$. Hence,
since $\epsilon$ is constant by definition and since by~(P1) we have that $k$
is constant and $n$ is sufficiently large, we have that $\epsilon^{k+1} \eta
\le \epsilon^{2k-1} \ln^{-2} n$. We also have that $6\epsilon^{(k+1/2)i + 2k-1}
\ln^{-2} n \le \epsilon^{(k+1/2)(i+1)}$.  Using precondition~(ii) we thus get
that
\begin{eqnarray} 
\nonumber \sum_{v \in V} \epsilon^{k+1} \eta \deg_i(v)^2 & \le& \sum_{v \in V}
\epsilon^{2k-1} \deg_i(v)^2 \ln^{-2} n \\ 
\label{eq:h1} &\le& 2\epsilon^{(2k-1)(i+1)} \Delta^2n \ln^{-2} n +  k
\epsilon^{(k+1/2)(i+1)} m p^{-1/2}.
\end{eqnarray}
Next suppose that~(\ref{eq:iq}) doesn't hold.  Then $2\epsilon^{(2k-1)i + k+1}
\Delta^2 n \le 0.5 k \epsilon^{(k+1/2)(i+1)} m p^{-1/2}$ and $\eta = 1$.
Therefore, using precondition~(ii) and since $6\epsilon^{1/2} \le 0.5$, we get
\begin{eqnarray} 
\nonumber \sum_{v \in V} \epsilon^{k+1} \eta \deg_i(v)^2 &\le& 0.5 k
\epsilon^{(k+1/2)(i+1)} m p^{-1/2} + 6 k \epsilon^{(k+1/2)i + k+1} m p^{-1/2}
\\ 
\label{eq:h2} &\le& k \epsilon^{(k+1/2)(i+1)} m p^{-1/2}.
\end{eqnarray}
We conclude that~(\ref{eq:todo}) is valid since
either~(\ref{eq:h1})~or~(\ref{eq:h2}) hold.
\end{proof}

In view of Proposition~\ref{claim:2}, it remains to upper bound the probability
that $Y$ deviates from its expectation by more than 
\begin{eqnarray*}
t_2 := \epsilon^{(2k-1)(i + 1)} \Delta^2 n \ln^{-2} n + k
\epsilon^{(k+1/2)(i+1)} m p^{-1/2}.
\end{eqnarray*}

Recall that the outcome of a vertex $v \in V_i$ is either the event that $v \in
V_{i+1}$ or not.  Clearly $Y$ depends on the outcomes of the vertices in $V_i$.
Let $a_v$ be the minimal integer so that if we change the outcome of $v \in
V_i$ then we can change $Y$ by at most an additive factor of $a_v$.
\begin{proposition} \label{claim:3} 
For all $v \in V_i$, $a_v \le 4 k \cdot \max\{2\epsilon^{(k-1)i}\Delta,
\Gamma\} \cdot \deg_i(v)$.
\end{proposition}
\begin{proof}
Let $v \in V_i$.  If $v \notin V_{i+1}$ then $\deg_{i+1}(v) = 0$ and otherwise
$\deg_{i+1}(v) \le \deg_i(v)$. Hence, changing the outcome of~$v$ can change
$\deg_{i+1}(v)^2$ by at most an additive factor of $\deg_i(v)^2$. Now let $u
\ne v$ be a vertex such that $\{v, u, \ldots\} \in E_i$.  Changing the outcome
of~$v$ can change $\deg_{i+1}(u)$ by at most an additive factor
of~$\codeg_i(u,v)$.  Since $\deg_{i+1}(u) \le \deg_i(u)$, this implies that
changing the outcome of~$v$ can change $\deg_{i+1}(u)^2$ by at most an additive
factor of $(\deg_i(u)+\codeg_i(u,v))^2 - \deg_i(u)^2 = 2 \codeg_i(u,v)
\deg_i(u) + \codeg_i(u,v)^2$. Lastly note that changing the outcome of $v$ can
affect only the sum $\deg_{i+1}(v)^2 + \sum_u \deg_{i+1}(u)^2$, where $\sum_u$
ranges over all $u \ne v$ such that $\{v, u, \ldots \} \in E_i$.  From the
above discussion we get that
\begin{eqnarray}
\nonumber a_v &\le& \deg_i(v)^2 + 2 \sum_{u} \codeg_i(u,v) \deg_i(u) +
\codeg_i(u,v)^2 \\ 
\label{eq:j1} &\le& \deg_i(v)^2 + 4 \sum_{u} \codeg_i(u,v) \deg_i(u). 
\end{eqnarray}
The proposition now follows from~(\ref{eq:j1}), precondition~(iii) and the fact
that
$$\sum_{u} \codeg_i(u,v) \le (k-1)\deg_i(v).$$
\end{proof}
From McDiarmid's inequality and Proposition~\ref{claim:3},
\begin{eqnarray} \nonumber \prob{|Y - \expec(Y)| > t_2} &\le& 2
\exp\bigg(-\frac{t_2^2}{\sum_{v \in V_i} a_v^2 }\bigg) \\ \label{eq:q2} &\le& 2
\exp\bigg(-\frac{t_2^2}{16 k^2 \cdot \max\{ 4 \epsilon^{2 (k-1)i} \Delta^2,
\Gamma^2\} \cdot \sum_{v \in V_i} \deg_i(v)^2 }\bigg).  \end{eqnarray} 
We complete the proof by considering the following four cases.
\begin{itemize}
\item
Assume that~(\ref{eq:iq}) holds and $4 \epsilon^{2(k-1)i} \Delta^2 \ge
\Gamma^2$.  From~(\ref{eq:iq}) and precondition~(ii) we have that $\sum_{v \in
V_i} \deg_i(v)^2$ is at most $12 k \epsilon^{(2k-1)i} \Delta^2 n\ln n$.  Also,
trivially $t_2 \ge \epsilon^{(2k-1)(i+1)} \Delta^2 n \ln^{-2}n$.  Therefore,
from~(\ref{eq:q2}) it follows that
\begin{eqnarray*}
\nonumber \prob{|Y - \expec(Y)| > t_2} &\le& 2
\exp\bigg(-\frac{\epsilon^{2(2k-1)(i+1)} \Delta^4 n^2 \ln^{-4}n} {1000 k^3
\epsilon^{2(k-1)i + (2k-1)i} \Delta^4 n \ln n }\bigg) \\ \nonumber &=& 2
\exp\bigg(-\frac{\epsilon^{i + 4k-2} n} {1000 k^3 \ln^5 n }\bigg) \\
\label{eq:l1} &\le& 2 \exp\bigg(-\frac{b_k p n} {\ln^5 n }\bigg),
\end{eqnarray*}
where the last inequality follows since $p  <  \epsilon^i$.
\item
Assume that~(\ref{eq:iq}) holds and $4 \epsilon^{2(k-1)i} \Delta^2 < \Gamma^2$.
We apply the same upper bound on $\sum_{v \in V_i} \deg_i(v)^2$ and the same
lower bound on $t_2$ given in the previous item together with~(\ref{eq:q2}) to
get
\begin{eqnarray*}
\nonumber \prob{|Y - \expec(Y)| > t_2} &\le& 2
\exp\bigg(-\frac{\epsilon^{2(2k-1)(i+1)} \Delta^4 n^2 \ln^{-4}n} {1000 k^3
\Gamma^2 \epsilon^{(2k-1)i} \Delta^2 n \ln n }\bigg) \\ \nonumber &=& 2
\exp\bigg(-\frac{\epsilon^{(2k-1)i + 4k-2} \Delta^2 n } {1000 k^3 \Gamma^2
\ln^5 n }\bigg) \\ 
 \label{eq:l2} &\le& 2 \exp\bigg(-\frac{b_k p^k m} { \Gamma^2 \ln^6 n }\bigg),
\end{eqnarray*}
where the last inequality follows from the assumption that~(\ref{eq:iq}) holds
and since $p  <  \epsilon^i$.
\item
Assume that~(\ref{eq:iq}) doesn't hold and $4 \epsilon^{2(k-1)i} \Delta^2 \ge
\Gamma^2$.  This gives us using precondition~(ii) that $\sum_{v \in V_i}
\deg_i(v)^2 \le 12 k \epsilon^{(k+1/2)i} m p^{-1/2}$.  Since $t_2 \ge k
\epsilon^{(k+1/2)(i+1)} m p^{-1/2} \ge k \epsilon^{k(i+1)} m$, we thus get
from~(\ref{eq:q2}) that
\begin{eqnarray*} 
\nonumber \prob{|Y - \expec(Y)| > t_2} &\le& 2 \exp\bigg(- \frac{k^2
\epsilon^{2k(i+1)} m^2}{1000 k^3 \epsilon^{2(k-1)i + (k+1/2)i} \Delta^2 m
p^{-1/2})}\bigg) \\
 \nonumber &=& 2 \exp\bigg(- \frac{m}{1000 k \epsilon^{(k-3/2)i - 2k} \Delta^2
p^{-1/2} }\bigg) \\
 \label{eq:l3} &\le& 2 \exp\Big(- b_k p n \Big), \end{eqnarray*}
where the last inequality follows from the assumption that~(\ref{eq:iq})
doesn't hold.
\item
Assume that~(\ref{eq:iq}) doesn't hold and $4 \epsilon^{2(k-1)i} \Delta^2 <
\Gamma^2$. The same bounds on $\sum_{v \in V_i} \deg_i(v)^2$ and on $t_2$ as in
the previous item hold.  Hence, from~(\ref{eq:q2}) it follows that
\begin{eqnarray*} \nonumber \prob{|Y - \expec(Y)| > t_2} &\le& 2 \exp\bigg(-
\frac{k^2 \epsilon^{2k(i+1)} m^2}{1000 k^3 \Gamma^2 \epsilon^{(k+1/2)i} m
p^{-1/2})}\bigg) \\
 \nonumber &=& 2 \exp\bigg(- \frac{\epsilon^{(k-1/2)i + 2k} m}{1000 k \Gamma^2
p^{-1/2} }\bigg) \\
 \label{eq:l4} &\le& 2 \exp\bigg(- \frac{b_k p^k m}{\Gamma^2}\bigg),
\end{eqnarray*}
where the last inequality follows since $p  <  \epsilon^i$. 
\end{itemize}
We conclude from the above that the second consequence holds with probability
at least $1 - \gamma_2$.
%

%%%%%%%%%%%%%%%%%%%%%%%%%%%%%%%%%%%%%%%%%%%%%%%%%%%%%%
\section{Proof of Theorem~\ref{thm:sub}} \label{sec:4}
%%%%%%%%%%%%%%%%%%%%%%%%%%%%%%%%%%%%%%%%%%%%%%%%%%%%%%
%
Let $G$ be a fixed complete or complete bipartite graph, with $v_G$ vertices
and $e_G \ge 3$ edges.  Assume that $H = H_G$, where $H_G$ is the hypergraph
that was defined in the introduction. Note that $n = \binom{N}{2}$, $m =
\Theta_G(n^{v_G})$, $\Delta = \delta = \Theta_G(n^{v_G-2})$, $\Delta_2 =
\Theta_G(n^{v_G-3})$ and $k = e_G$.  Let $X_G = X$.  Recall that $\rho_1 =
v_G/e_G$ and $\rho_2 = (v_G-2)/(e_G-1)$.  Fix a positive constant $c_1$ and let
$c_2 := 0.1 c_1 / (e_G + c_1)$.   Assume that $N^{-\rho_1 + c_1} \le p \le
N^{-\rho_2 - c_1}$, $8 \ln N \le \lambda \le N^{c_2}$ and let $\Gamma :=
N^{c_1}$.  We prove that~(\ref{eq:q}) holds.  For that, it is safe to
assume that $n \ge n_0$ for a sufficiently large constant $n_0$.

The next lemma is proved below.
\begin{lemma} \label{lemma:nice} 
$(H, p, 0.25 \lambda, \Gamma, 0.5 b)$ is nice for some positive constant $b$.  \end{lemma}
From Lemma~\ref{lemma:nice}, Theorem~\ref{thm:main}, the assumptions on $H, p$
and $\lambda$ above and the definition of $\Gamma$, it easily follows that
\begin{eqnarray*} \label{eq:4a} 
\prob{|X_G - \expec(X_G)| > (\ln n + 0.25 \lambda) \expec(X_G)^{1/2}} \le
e^{-c_G \lambda^2},
\end{eqnarray*}
where $c_G$ is a positive constant that depends only on $G$.  In addition we
have that $\ln n \le 2 \ln N \le 0.25 \lambda$ and that $0.5 \expec(X_G)^{1/2}
< \var(X_G)^{1/2}$ for our choice of $p$.  Thus,
\begin{eqnarray*} 
 \prob{|X_G - \expec(X_G)| \ge \lambda \var(X_G)^{1/2} } &\le& \prob{|X_G -
\expec(X_G)| > 0.5 \lambda \expec(X_G)^{1/2} }  \\
&\le& \prob{|X_G - \expec(X_G)| > (\ln n  +0.25 \lambda) \expec(X_G)^{1/2} } \\
&\le& e^{-c_G \lambda^2}. \end{eqnarray*}

All that remains is to prove Lemma~\ref{lemma:nice}. For that we need to show
that the four properties given in Definition~\ref{def:nice} hold.
Property~(P1)~holds since $G$ is a fixed graph with $e_G \ge 3$, $p \le
N^{-\rho_2 - c_1}$ and $n = \binom{N}{2} \ge n_0$.  Property~(P2)~follows from
our assumed lower bound on $p$ and the fact that $k = e_G \ge 3$, which give
$(p^km)^{1/2} \ge N^{c_1} \ge \ln n$, together with the fact that $c_2 \le
c_1$, which gives $\lambda \le N^{c_2} \le N^{c_1}$.  Property~(P3)~holds since
$v_G \ge 3$ and $\delta = \Theta_G(N^{v_G-2})$ while $\Delta_2 =
\Theta_G(N^{v_G-3})$. It remains to show that property~(P4)~holds.  This is
done in the next subsection.

\subsection{Property~(P4)}
If $v_G = 3$ then $G$ is a triangle and in that case, property~(P4) can be
easily shown to hold using Chernoff's bound. For the rest of this section we
assume that $v_G \ge 4$. 

A {rooted graph} $(R,F)$ is a graph $F$ whose vertex set is the union of two
disjoint sets of labeled vertices,  $R$ and $S$, where the vertices in $R$ are
called the {roots} and $|S| > 0$.  In such a rooted graph, let $\{x_1, x_2,
\ldots, x_r\}$ be the labels of $R$ and let $\{y_1, y_2, \ldots, y_s\}$ be the
labels of $S$.
The {density} of such a rooted graph is defined to be the ratio $t/s$, where $t$
is the number of edges in $F$  excluding the edges induced by $R$. We say that
$(R,F)$ is {balanced} if for every induced subgraph $F' \subseteq F$ that
contains the vertex set $R$ and which has at least $r + 1$ vertices, the
density of $(R,F')$ is not larger than the density of $(R, F)$.

Let $(R,F)$ be a rooted graph as above (in particular, $|R| = r$ and $F$ has
$r+s$ vertices). Let $R'$ be a set of $r$ vertices in $K_N$ that are labeled by
$\{x_1', x_2', \ldots, x_r'\}$. Let $S'$ be a set of $s$ vertices disjoint from
$R'$ in $K_N$. Let $F'$ be a subgraph of $K_N$ over the vertex set $R' \cup S'$
with $R'$ being an independent set in $F'$.  We say that $F'$ is an {extension}
with respect to $(R,F)$ and $R'$ if one can label the vertices in $S'$ by
$\{y_1', y_2', \ldots, y_s'\}$ so that (identifying vertices with their
labels):
\begin{itemize}
\item $\{x_i, y_j\} \in F$ if and only if $\{x_i', y_j'\} \in F'$;
\item $\{y_i, y_j\} \in F$ if and only if $\{y_i', y_j'\} \in F'$.
\end{itemize}
Let $Z_{(R,F), R'}$ be the random variable that counts the number of
extensions with respect to a given rooted graph $(R,F)$ and a given set of
labeled vertices $R'$ that are contained in $G(N, q)$. The next result follows
from Corollary~6.7~in~\cite{Vu02}.
\begin{theorem} \label{thm:kv}
There is a positive constant $b$ such that the following holds.  Let $(R,F)$ be
a balanced rooted graph.  If $\expec(Z_{(R,F), R'}) \ge N^{0.2c_1}$ then 
$$\prob{|Z_{(R,F), R'} - \expec(Z_{(R,F), R'})| \ge 0.5 \expec(Z_{(R,F), R'})}
\le e^{-b N^{0.2c_1/e_G}}.$$
\end{theorem}

Label the vertices of $G$ with $\{x_1, x_2, \ldots, x_{v_G}\}$. Identify each
vertex of $G$ with its label and assume without loss of generality that $\{x_1,
x_2\} \in G$.  Let $(R_1, G)$ and $(R_2, G)$ be two rooted graphs, where $R_1 =
\{x_1, x_2\}$ and $R_2 = \{x_1, x_2, x_3\}$.  It is easy to verify that $(R_1,
G)$ and $(R_2, G)$ are balanced.  Let $R_1'$ be a set of two vertices in $K_N$,
labeled with $\{x_1', x_2'\}$.  Let $R_2'$ be a set of three vertices in $K_N$,
labeled with $\{x_1', x_2', x_3'\}$.  Let $Z_1 := Z_{(R_1, G), R_1'}$ and $Z_2
:= Z_{(R_2, G), R_2'}$.  
Observe that conditioned on $e_1 \in V_q$, where $e_1$ is the edge that is
induced by $R_1'$, we have $\deg_q(e_1) = Z_1$. Furthermore, if for every
choice of $R_2'$ and its labelling we have $Z_2 \le z$ then the maximum co-degree of $H_q$ is at
most $2z$. 
The following lemma together with a union bound argument gives property~(P4).
\begin{lemma} \label{lemma:z}
Let $p \le q < 1$.
There is a positive constant $b$ such that with probability at least $1 - 3
e^{-b \lambda^2}$,
\begin{enumerate}
\item $Z_1 \le \max\{2q^{k-1} \Delta, \Gamma\}$;
\item $p^{1/2} q^{k-3/2} \Delta^2 n \ln n \ge m \implies Z_2 \le Z_1 \ln^{-4}
n$.
\end{enumerate}
\end{lemma}
\begin{proof}
Let $b$ be the positive constant implied to exist by Theorem~\ref{thm:kv}. We
show that the first item holds with probability at least $1 - e^{-b\lambda^2}$
while the second item holds with probability at least $1 - 2e^{-b\lambda^2}$.

Let $c_0$ be the minimal positive real for which it holds that $q \ge
N^{-\rho_2 + c_0}$ implies $\expec(Z_1) \ge 0.5 \Gamma = 0.5 N^{c_1}$.  If $q
\ge N^{-\rho_2 + c_0}$ then we are done since by Theorem~\ref{thm:kv} and the
fact that $\lambda^2 \le N^{0.2c_1/e_G}$, we have that with probability at
least $1 - e^{-b \lambda^2}$, $Z_1 \le 2 \expec(Z_1) = 2q^{k-1}\Delta$.  Note
that in particular, if $q = N^{-\rho_2 + c_0}$ then with probability at least
$1 - e^{-b \lambda^2}$ we have $Z_1 \le 2 \expec(Z_1) = \Gamma$.  Thus, by a
monotonicity argument we get that if $q < N^{-\rho_2 + c_0}$ then with
probability at least $1 - e^{-b \lambda^2}$, $Z_1 \le \Gamma$.

Assume  $q$ satisfies $p^{1/2} q^{k-3/2} \Delta^2 n \ln n \ge m$.  This implies
that $q \ge N^{-\rho_2 + 0.5 c_1 / e_G}$. From this we infer that $\expec(Z_1)
\ge N^{0.3c_1}$. 
Keeping that in mind, if $q$ is such that $\expec(Z_2) = N^{0.2c_1}$ then it
follows from Theorem~\ref{thm:kv} that with probability at least $1 - 2 e^{- b
\lambda^2}$, $Z_2 \le Z_1 \ln^{-4} n$.  By a monotonicity argument we can reach
the same conclusion for every $q$ for which it holds that $\expec(Z_2) \le
N^{0.2c_1}$.
Next note that $\expec(Z_1) = \Theta_G(N^{v_G-2} q^{e_G-1})$.  If $G$ is a
complete graph then $\expec(Z_2) = \Theta_G(N^{v_G-3}q^{e_G-3})$ and $q \ge
N^{-\rho_2 + 0.5c_1/e_G} \ge N^{-1/2 + 0.5 c_1 / e_G}$, which implies
$\expec(Z_2) \le \expec(Z_1) N^{-c_2}$.
If on the other hand $G$ is a complete bipartite graph then $\expec(Z_2) =
\Theta_G(N^{v_G-3}q^{e_G-2})$ and $q \ge N^{-\rho_2 + 0.5c_1/e_G} \ge N^{-1 +
0.5c_1 / e_G}$, which again implies $\expec(Z_2) \le \expec(Z_1) N^{-c_2}$.
Thus, if $q$ is such that $\expec(Z_2) > N^{0.2c_1}$ then by
Theorem~\ref{thm:kv} we get that with probability at least $1 - 2 e^{-b
\lambda^2}$, $Z_2 \le Z_1 \ln^{-4} n$ as needed.  
\end{proof}

%%%%%%%%%%%%%%%%%%%%%%%%%%%%%%%%%%%%%%%%%%%%%%%%%%%%%%%%%%%%%%%%%%%%%%
%%%%%%%%%%%%%%%%%%%%%%%%%%%%%%%%%%%%%%%%%%%%%%%%%%%%%%%%%%%%%%%%%%%%%%
% Bibliography
%%%%%%%%%%%%%%%%%%%%%%%%%%%%%%%%%%%%%%%%%%%%%%%%%%%%%%%%%%%%%%%%%%%%%%
%%%%%%%%%%%%%%%%%%%%%%%%%%%%%%%%%%%%%%%%%%%%%%%%%%%%%%%%%%%%%%%%%%%%%%

\end{document}